\documentclass[a4paper,leqno]{amsart}

\usepackage{amsmath,amssymb,amsthm}
\usepackage{hyperref}

\def\H{{\mathbb H}}
\def\R{{\mathbb R}}
\def\N{{\mathbb N}}

\def\Sph{{\mathbb S}} 
\def\F{\mathcal F}

\def\O{\mathcal O}
\def\L{\mathcal L}

\def\virgp{\raise 2pt\hbox{,}}

\def\({\left(}
\def\){\right)}
\def\<{\left\langle}
\def\>{\right\rangle}
\def\le{\leqslant}
\def\ge{\geqslant}

\def\Eq#1#2{\mathop{\sim}\limits_{#1\rightarrow#2}}
\def\Tend#1#2{\mathop{\longrightarrow}\limits_{#1\rightarrow#2}}

\def\d{{\partial}}

\def\eps{\varepsilon}

\def\si{{\sigma}}

\DeclareMathOperator{\RE}{Re}

\theoremstyle{plain}
\newtheorem{theorem}{Theorem}[section]
\newtheorem{lemma}[theorem]{Lemma}
\newtheorem{corollary}[theorem]{Corollary}
\newtheorem{proposition}[theorem]{Proposition}

\theoremstyle{definition}

\newtheorem{notation}[theorem]{Notation}

\newtheorem{remark}[theorem]{Remark}
\newtheorem*{remark*}{Remark}

\numberwithin{equation}{section}

\begin{document}

\title[Scattering for NLS in different geometries]{On scattering for
  NLS: from Euclidean to hyperbolic space}   
\author[V. Banica]{Valeria Banica}
\address[V. Banica]{D\'epartement de Math\'ematiques\\ Universit\'e
  d'Evry\\ Bd. F.~Mitterrand\\ 91025 Evry\\ France} 
\email{Valeria.Banica@univ-evry.fr}
\author[R. Carles]{R\'emi Carles}
\address[R. Carles]{CNRS \& Universit\'e Montpellier~2\\Math\'ematiques
\\CC~051\\Place Eug\`ene Bataillon\\34095
  Montpellier cedex 5\\ France}
\email{Remi.Carles@math.cnrs.fr}
\author[T. Duyckaerts]{Thomas Duyckaerts}
\address[T. Duyckaerts]{D\'epartement de Math{\'e}matiques\\ Universit{\'e} de Cergy-Pontoise\\CNRS UMR 8088\\
2 avenue Adolphe Chauvin\\ BP 222, Pontoise\\ 95302 Cergy-Pontoise
cedex\\ France}
\email{tduyckae@math.u-cergy.fr}
\begin{abstract}
We prove asymptotic completeness in the energy space for the
nonlinear Schr\"odinger equation posed on hyperbolic space $\H^n$ in
the radial case, for
$n\ge 4$, and any energy-subcritical, defocusing, power
nonlinearity. The proof is based on simple 
Morawetz estimates and weighted Strichartz estimates. We investigate
the same question on spaces which kind of 
interpolate between Euclidean space and hyperbolic space, showing
that the family of short range nonlinearities becomes larger and
larger as the space approaches the hyperbolic space. Finally, we
describe the large time behavior of radial solutions to the free
dynamics. 
\end{abstract}
\thanks{V.B. is partially
  supported by the ANR project ``\'Etude qualitative des E.D.P.'',
R.C., by the ANR project SCASEN, and  
T.D., by the ANR project ONDNONLIN}
\maketitle


\section{Introduction}
Consider the defocusing nonlinear Schr\"odinger equation on Euclidean
space
\begin{equation}
  \label{eq:nlsEucl}
i\d_t u + \Delta u= |u|^{2\sigma}u,\quad x\in\R^n, n\ge 3\quad ;\quad
u_{\mid t=0}=u_0\in H^1(\R^n),
\end{equation}
where $\Delta$ stands for the usual Laplacian.  For
$0<\si\le 2/(n-2)$, the solution to \eqref{eq:nlsEucl} is global in
time, in the class of finite energy solutions
\cite{GV79Cauchy,CKSTTAnnals,RV,VisanH1}. If in addition $\si>2/n$,
then there is scattering in $H^1$ \cite{GV85,CKSTTAnnals,RV,VisanH1}:
\begin{equation*}
  \exists u_\pm \in H^1(\R^n),\quad \left\lVert
  u(t)-e^{it\Delta}u_\pm\right\rVert_{H^1}\Tend t {\pm \infty}0 . 
\end{equation*}
On the other hand, if $\si$ is too small, then long range effects are
present, and the above result holds only in the trivial case
\cite{Barab,Strauss74}: if $\si\le 1/n$ and $u_+\in L^2(\R^n)$, $u\in
C(\R;L^2(\R^n))$ are such that
\begin{equation*}
  \left\lVert
  u(t)-e^{it\Delta}u_+\right\rVert_{L^2}\Tend t {+ \infty}0 ,
\end{equation*}
then necessarily $u_+=u=0$ (even if the functions are supposed to be
radial). In other words, linear and nonlinear 
dynamics are not comparable for large time if $\si\le 1/n$. In this
paper, we show that this phenomenon disappears for radial solutions,
when the space variable 
belongs to the hyperbolic space instead of the Euclidean space. Such a
phenomenon was established in \cite{BCS} in the three-dimensional
case, with partial results in other dimensions. We prove asymptotic
completeness in the case of higher dimensions. Moreover, we consider
rotationally symmetric manifolds, which may be viewed as
interpolations between 
Euclidean and hyperbolic spaces, as introduced in \cite{BD}. We show
that asymptotic completeness holds for radial solutions and
$\si_0(n)<\si$, for some explicit value $\si_0(n)$, going to zero as
the space approaches 
the hyperbolic space. The proof relies on simple Morawetz estimates
(as opposed to interaction Morawetz estimates, as introduced in
\cite{CKSTTCPAM}), and weighted Strichartz estimates
\cite{VittoriaDR,BD}.  
The energy-critical case is not
considered: we always assume $\si<2/(n-2)$. 

We begin with the nonlinear
Schr\"odinger equation on hyperbolic space 
\begin{equation}
  \label{nls}
i\d_t u + \Delta_{\mathbb{H}^n} u= |u|^{2\sigma}u,\quad x\in\H^n\quad ;\quad
u_{\mid t=0}=u_0\in H^1(\H^n),
\end{equation}
where $x=(\cosh r,\omega \sinh r)\in\H^n\subset \R^{n+1}$, $r\ge 0$,
$\omega \in\Sph^{n-1}$ and 
$$\Delta_{\H^n}=\partial^2_r+(n-1)\frac{\cosh r}{\sinh
  r}\partial_r+\frac{1}{\sinh^2 r}\Delta_{\mathbb{S}^{n-1}}.$$ 
In \cite{BCS} it has been proved that for small radial initial data,
there is asymptotic completeness in $L^2$ for all
$0<\sigma<2/n$ and $n\ge 2$. At the $H^1$ level, wave
operators were proved to exist for all $0<\sigma<2/(n-2)$ and
$n\ge 2$, without restriction on the size of the radial data. The
main ingredient were radial Strichartz estimates similar to those used
on $\R^d$, with arbitrary 
$d\ge n$ (so the assumption $\si>2/d$ on $\R^d$ boils down to $\si>0$,
since $d\ge n$ is arbitrary). Such estimates stem from weighted
Strichartz estimates on 
$\H^n$ in the radial case (see \cite{BaHyper} for $n=3$,
\cite{VittoriaDR} for $n\ge 4$, and \cite{BCS} for $n=2$). Finally,
asymptotic completeness was proved for all $0<\sigma<2/(n-2)$, without
restriction on the size of the radial data, but only in dimension
$n=3$. The latter result used in addition interaction Morawetz
estimates, valid also in the non-radial case, in all dimensions $n\ge
3$. The issue for $n\ge 4$ was that the passage from the interaction
Morawetz estimates to global in time estimates in mixed spaces is done
\emph{via} a delicate Fourier argument on $\R^n$ \cite{TaoVisanZhang},
difficult to adapt to 
hyperbolic space. Moreover, the historical approach based on simple
Morawetz estimates relies on a precise dispersive rate for the free
Schr\"odinger group (see
e.g. \cite{CazCourant}), which is not known on $\H^n$ for $n\ge 4$.
In this paper we cover the cases $n\ge 4$ by using
simple Morawetz estimate and weighted Strichartz estimates.  
We  focus on the radial case.
\begin{remark}
 Quite simultaneously to this work, the existence of scattering
 operators in $H^1(\H^n)$ for $n\ge 2$ and $0<\si<2/(n-2)$ was
 established in \cite{IS08} (see also \cite{AnPi}), without the radial
 symmetry assumption 
 that we make in this paper. The authors have derived new Morawetz
 estimates, which overcome the difficulties pointed out above, thanks
 also to new Strichartz estimates. Our
 point of view in the present paper is 
 rather to insist on the transition between Euclidean to hyperbolic
 geometry, as explained below. Also, the proof of the asymptotic
 completeness in the radial case is naturally shorter, and serves as a
 basis to study the case of intermediary metrics, where in addition no
 Fourier analysis seems to available. 
\end{remark}

\begin{theorem}\label{theo:CA}
Let $n\ge 4$ and 
$$0<\sigma<\frac{2}{n-2}.$$
Then asymptotic completeness holds in $H_{\rm rad}^1(\H^n)$ for
\eqref{nls}: for all $u_0\in H_{\rm rad}^1(\H^n)$, there exists
$u_+\in H_{\rm rad}^1(\H^n)$ such that
\begin{equation*}
\left\lVert
  u(t)-e^{it\Delta_{\H^n}}u_+\right\rVert_{H^1(\H^n)} \Tend t {+\infty}0 ,
\end{equation*}
where $u$ is the solution to \eqref{nls}. 
\end{theorem}
In view of \cite{BCS}, the wave operators $W_\pm$ are well-defined on
$H_{\rm rad}^1(\H^n)$ for this range of $\si$. The above result shows that
the wave operators are invertible on $H_{\rm rad}^1(\H^n)$
($u_+=W_+^{-1}u_0$), so we infer the existence
of a scattering operator for arbitrarily large data, with no long
range effect. This extends the 
result of \cite{BCS}, established for $n=3$ only.
\begin{corollary}\label{cor:scatt}
  For $n\ge 3$, and 
  \begin{equation*}
    0<\si<\frac{2}{n-2},
  \end{equation*}
the scattering operator $S=W_+^{-1}W_-$ associated to
  \eqref{nls} is well-defined from $H_{\rm rad}^1(\H^n)$ to
  $H_{\rm rad}^1(\H^n)$: for all $u_-\in H_{\rm rad}^1(\H^n)$, there exists
  $u\in C(\R; H_{\rm rad}^1(\H^n))$ solution to 
  \begin{equation*}
    i\d_t u + \Delta_{\H^n} u = \lvert u\rvert^{2\si}u,
  \end{equation*}
such that
\begin{equation*}
  \left\lVert u(t)- e^{it\Delta_{\H^n}}u_-\right\rVert_{H^1(\H^n)} \Tend
  t{-\infty }0, 
\end{equation*}
and a unique $u_+=Su_- \in H_{\rm rad}^1(\H^n)$ such that
\begin{equation*}
\left\lVert
  u(t)-e^{it\Delta_{\H^n}}u_+\right\rVert_{H^1(\H^n)} \Tend t {+\infty}0 .
\end{equation*}
\end{corollary}
The absence of long range effects is of course an effect of the
geometry of the hyperbolic space. Typically, the usual algebraic decay
on $\R^n$ is replaced by an exponential decay. This vague statement
can be compared to the phenomenon studied in \cite{CaSIMA}, where
instead of changing the geometry of the space, an external potential
was added:
\begin{equation*}
  i\d_t u +\Delta u = -\lvert x\rvert^2 u + \lvert
  u\rvert^{2\si}u,\quad x\in\R^n\quad ;\quad u_{\mid t=0}=u_0\in
  \Sigma=H^1\cap \F(\H^1).
\end{equation*}
The effect of this repulsive harmonic potential (as opposed to the
usual harmonic potential $+\lvert x\rvert^2$) is to accelerate the
particle which goes to infinity exponentially fast, so that asymptotic
completeness holds in $\Sigma$ for any $0<\si<2/(n-2)$ (no long range
effect). In \cite{BCHM} \emph{linear} scattering theory was considered
for perturbations of the Hamiltonian $-\Delta-\lvert
x\rvert^{\alpha}$, for $0<\alpha\le 2$. It is shown that the
borderline between short range and long range moves as $\alpha$ varies
from $0$ to $2$. Essentially, a potential $V$ is short range as soon
as $|V(x)|\lesssim \<x \>^{-1+\alpha/2-\eps}$ when $\alpha <2$, and 
$|V(x)|\lesssim (1+\log\<x \>)^{-1-\eps}$ when $\alpha=2$, for some
$\eps>0$; the dynamics generated by $-\Delta-\lvert 
x\rvert^{\alpha}$ accelerates the particles, from an algebraic decay
with a larger and larger power, to the limiting exponential case (if
$\alpha>2$, the underlying operator is not even essentially
self-adjoint on $C_0^\infty(\R^n)$, due to infinite speed of
propagation, see e.g. \cite{Dunford}). Note
that nonlinear perturbations of $-\Delta-\lvert
x\rvert^{\alpha}$ for $0<\alpha<2$ have not been studied, due to a
lack of suitable technical tools. In the present paper, we analyze
what can be considered as the geometrical counterpart of this
problem. 
\begin{notation}
  Let $k\in\N$ and 
  \begin{equation*}
    \phi(r)=\sum_{j=0}^k \frac{1}{(2j+1)!}r^{2j+1} . 
  \end{equation*}
We denote by $M_k^n$ (or simply $M$ when there is no possible
confusion) the $n$-dimensional rotationally symmetric
manifold with metric
\begin{equation*}
  ds^2 = dr^2 +\phi(r)^2 d\omega^2,
\end{equation*}
where $d\omega^2$ stands for the metric on $\Sph^{n-1}$. 
\end{notation}
The Laplace--Beltrami operator on $M_k^n$ is 
\begin{equation*}
  \Delta_M= \d_r^2 +(n-1)\frac{\phi'(r)}{\phi(r)}\d_r
  +\frac{1}{\phi(r)^2}\Delta_{\Sph^{n-1}}. 
\end{equation*}
\begin{remark}
  If $k=0$, we recover the Euclidean case. The hyperbolic
  case corresponds to $k=\infty$. The manifold $M_k^n$ can thus be
  viewed as an interpolation between these two cases. 
\end{remark}
\begin{theorem}\label{theo:scattM}
  Let $n\ge 4$. For $k\in \N$, consider the nonlinear Schr\"odinger
  equation
  \begin{equation}
    \label{eq:nlsM}
    i\d_t u + \Delta_{M} u= |u|^{2\sigma}u,\quad x\in M_k^n\quad ;\quad
u_{\mid t=0}=u_0\in H^1_{\rm rad}(M_k^n). 
  \end{equation}
Set $N= (2k+1)(n-1)+1$. For $2/N<\si<2/(n-2)$, asymptotic completeness
holds in $H^1_{\rm rad}(M)$: for all $u_0\in H^1_{\rm
  rad}(M)$, there exists $u_+ \in H^1_{\rm rad}(M)$ such that
\begin{equation*}
  \left\lVert u(t) -e^{it\Delta_M}u_+\right\rVert_{H^1(M)}\Tend t
  {+\infty}0. 
\end{equation*}
\end{theorem}
From the above example, we see that this result is a
transition between Euclidean ($k=0$) and hyperbolic ($k\to \infty$) cases. In
view of the results of \cite{BD}, we infer
\begin{corollary}
    For $n\ge 4$, $k\in \N$, $N= (2k+1)(n-1)+1$ and 
  \begin{equation*}
    \frac{2}{N}<\si<\frac{2}{n-2}, 
  \end{equation*}
the scattering operator $S=W_+^{-1}W_-$ associated to
  \eqref{eq:nlsM} is well-defined from $H_{\rm rad}^1(M)$ to
  $H_{\rm rad}^1(M)$: for all $u_-\in H_{\rm rad}^1(M)$, there exists
  $u\in C(\R; H_{\rm rad}^1(M))$ solution to 
  \begin{equation*}
    i\d_t u + \Delta_{M} u = \lvert u\rvert^{2\si}u,
  \end{equation*}
such that
\begin{equation*}
  \left\lVert u(t)- e^{it\Delta_{M}}u_-\right\rVert_{H^1(M)} \Tend
  t{-\infty }0, 
\end{equation*}
and a unique $u_+=Su_- \in H_{\rm rad}^1(M)$ such that
\begin{equation*}
\left\lVert
  u(t)-e^{it\Delta_{M}}u_+\right\rVert_{H^1(M)} \Tend t {+\infty}0 .
\end{equation*}
\end{corollary}
\begin{remark}
  We see that as soon as $k\ge 1$, $2/N<1/n$. In view of the results
  of \cite{Barab,Strauss74}, this shows that the curved geometry
  already changes the short range/long range borderline. In
  Section~\ref{sec:free}, we present a rather formal argument, relying
  on the description of the free dynamics (see
  Proposition~\ref{prop:free} below) 
  indicating that for $\si\le 1/N$, long range effects are present
  (see Remark~\ref{rem:long}).  
\end{remark}
\begin{remark}
  The proof we
  present still works for other functions $\phi$. We choose to
  restrict our attention to such spaces $M_k^n$ in order to emphasize
  the transition between Euclidean and hyperbolic spaces. 
\end{remark}
\begin{remark}
  The existence of a ``scattering'' dimension $N= (2k+1)(n-1)+1$ can
  be compared to Sobolev embeddings on the Heisenberg group. It is
  shown in \cite{BGX} that the indices for Sobolev embeddings on the
  $(2n+1)$-dimensional Heisenberg group correspond to their counterparts on
  $\R^{2n+2}$.  
\end{remark}
To conclude this introduction, and give a rather general picture of
large time dynamics of solutions to Schr\"odinger equations, we
describe the asymptotic behavior of the free dynamics in the radial
setting. It seems that the analogous result in the non-radial case,
even on hyperbolic space, is not available so far.  
\begin{proposition}\label{prop:free}
  Let $n\ge 2$. \\
$(1)$ Consider the linear
equation
 \begin{equation*}
i\d_t u + \Delta_{\mathbb{H}^n} u= 0,\quad x\in\H^n\quad ;\quad
u_{\mid t=0}=u_0\in L^2_{\rm rad}(\H^n).
\end{equation*}
There exists a linear operator $\L$, unitary from $L^2_{\rm
  rad}(\H^n)$ to $L^2_{\rm rad}(\R^n)$,  such that
\begin{align*}
  &\left\lVert u(t) -v(t)\right\rVert_{L^2(\H^n)}\Tend t {+\infty}
  0,\\
\text{where }& v(t,r) =
  \frac{e^{-i(n-1)t/2+ir^2/(4t)}}{t^{n/2}} \(\frac{r}{\sinh
  r}\)^{\frac{n-1}{2}} \(\L u_0\)\(\frac{r}{t}\). 
\end{align*}
$(2)$ Let $k\ge 1$. Consider the linear
equation
 \begin{equation*}
i\d_t u + \Delta_{M} u= 0,\quad x\in M^n_k\quad ;\quad
u_{\mid t=0}=u_0\in L^2_{\rm rad}(M^n_k).
\end{equation*}
There exists a linear operator $\L$, unitary from $L^2_{\rm
  rad}(M^n_k)$ to $L^2_{\rm rad}(\R^n)$, such that
\begin{align*}
  &\left\lVert u(t) -v(t)\right\rVert_{L^2(M)}\Tend t {+\infty}
  0,\\
\text{where }& v(t,r) =
  \frac{e^{ir^2/(4t)}}{t^{n/2}} \(\frac{r}{\phi
  (r)}\)^{\frac{n-1}{2}} \(\L u_0\)\(\frac{r}{t}\). 
\end{align*}
\end{proposition}
\begin{remark}
 In the Euclidean case $k=0$, $\L$ is, up to a multiplicative constant
and a dilation, the usual Fourier transform. In the case of $\H^3$,
the first point of Proposition~\ref{prop:free} was established in
\cite{BCS}. There again, $\L$ is essentially the Fourier transform. It
is not clear whether the same holds in the case of $\H^n$, for $n\not
= 3$, where Fourier analysis is well developed. See
Remark~\ref{rem:Fourier}. 
\end{remark}
The rest of the paper is organized as follows. In the next paragraph,
we recall the general approach for Morawetz inequalities, and give
applications for the case of defocusing nonlinear Schr\"odinger
equations on $\H^n$ or $M$. We prove Theorems~\ref{theo:CA} and
\ref{theo:scattM} in \S\ref{sec:CAH} and \S\ref{sec:scattM},
respectively. Proposition~\ref{prop:free} is established in
\S\ref{sec:free}.

\section{Morawetz inequality}
\label{sec:Morawetz}

We first recall the general virial computation on a manifold
$M$, where technical ingredients such as integration by parts work as
in the Euclidean case. Typically, $M$ can be chosen to be $\R^n$,
$\R^n\times \R^n$,  $\H^n$ or $M_k^n$, 
with no restriction on the dimension. The homogeneous contribution is
treated in \cite{HTW05}, and the inhomogeneous case is easily
inferred. 
\begin{lemma}[Virial inequality]  \label{viriel}
Let $a$ be a real function on $M$ with positive Hessian. 
If $v$ is a global  $L^\infty(\R,H^1(M))$ solution of
\begin{equation}\label{nlsgen} 
i\partial_t v + \Delta_Mv=Fv\quad ;\quad 
v_{\mid t=0}=v_0\in H^1(M),
\end{equation}
then there exists a positive constant $C$ such that
\begin{equation}\label{virielcomp}
  \begin{aligned}
    \int_0^T\Big(\int_M (-\Delta^2 a)\frac{|v|^2}{2}&+\RE\int_M 
 2 Fv\nabla\overline{v}\cdot \nabla a+\overline{F}|v|^2\Delta a
  \Big)\le\\
&\le 
C\sup_{t\in[0,T]}\int_M |\overline{v}\nabla v\cdot\nabla a|. 
  \end{aligned}
\end{equation}
\end{lemma}

\begin{lemma}[Morawetz inequality on $\H^n$] \label{simpleM}
Let $n\ge
 4$. All solutions $u$ of equation \eqref{nls} (not necessarily
 radial) satisfy to 
\begin{equation}\label{Mosimple}
\int_0^T\int_{\H^n} \frac{\cosh r}{\sinh^3r}|u(t,x)|^2dx\,dt\le
C\sup_{t\in[0,T]}\| u(t)\|_{H^1}^2,
\end{equation}
where $r=d_{\H^n}(O,x)$. 
\end{lemma}
\begin{proof}
We apply Lemma~\ref{viriel}, with $M=\H^n$, $u=v$, $F=|u|^{2\sigma}$ and
$a(x)=r=d_{\H^n}(O,x)$. In the left hand side of \eqref{virielcomp},
the contribution of  the nonlinearity is 
$${\mathcal N}=\int_0^T\left(\int_{\H^n}\frac{2}{2\sigma+2}\,
  \nabla(|u|^{2\sigma+2})\cdot \nabla a+|u|^{2\sigma+2}\Delta a \right),$$ 
so by integrating by parts the first term,
$${\mathcal N}=\int_0^T\left(\int_{\H^n}-\frac{1}{\sigma+1}\,
  |u|^{2\sigma+2}\Delta a+|u|^{2\sigma+2}\Delta a
\right)=\int_0^T\int_{\H^n}\frac{\sigma}{\sigma+1}\,
|u|^{2\sigma+2}\Delta a.$$ 
Since $\Delta a=(n-1)\frac{\cosh r}{\sinh r}$, the nonlinear
contribution is non-negative (defocusing nonlinearity), and we get
$$\int_0^T\int_M (-\Delta^2 a)\frac{|u|^2}{2}\le
C\sup_{t\in[0,T]}\|u(t)\|_{H^1}^2.$$ 
By computing
$\displaystyle -\Delta^2 a(x)=(n-1)(n-3)\frac{\cosh r}{\sinh^3r},$ 
the lemma follows. 
\end{proof}
\begin{lemma}[Morawetz inequality on $M$]\label{lem:MorawetzM}
  Let $n\ge 4$ and $k\in \N$. All solutions to \eqref{eq:nlsM} (not
  necessarily radial) satisfy to
  \begin{equation}
    \label{eq:MorawetzM}
    \int_0^T \int_M \frac{1}{r^3}\lvert u(t,x)\rvert^2 dxdt \le C
    \sup_{t\in [0,T]} \lVert u(t)\rVert_{H^1(M)}^2,
  \end{equation}
where $r=d_M(O,x)$. 
\end{lemma}
\begin{proof}
  The proof follows the same lines as above. For $k=0$, this is the
  standard Morawetz estimate; see e.g. \cite{CazCourant}. We therefore
  assume $k\ge 1$. The manifold $M_k^n$ has
  a negative sectional curvature, so the Hessian of the function
  distance to the origin is positive (Theorem 3.6 of \S6 in
  \cite{Pe}). For $a(x)=r= d_M(O,x)$,  we compute
  \begin{align*}
    \Delta_M a &= (n-1)\frac{\phi'}{\phi}, \\
\Delta^2_M a &= \frac{n-1}{\phi^3}\( \phi^2\phi^{(3)}
+(n-4)\phi\phi'\phi'' -(n-3)\(\phi'\)^3\). 
  \end{align*}
Since $\Delta_M a$ is non-negative, the nonlinear term is neglected,
just like in the proof of Lemma~\ref{simpleM}. 
We check
\begin{align*}
 -\Delta^2_M a &\Eq r 0  (n-1)(n-3)\frac{1}{r^3},\\
-\Delta^2_M a &\Eq r \infty (n-1)(2k+1)\( 2k(n-1)+
n-3\)\frac{1}{r^3}. 
\end{align*}
To establish the lemma, it suffices to prove that $-\Delta^2_M a>0$
for $r>0$. 
Write the numerator of $-\Delta^2_M a$ as 
\begin{equation*}
  (n-1)\((n-3)\phi'\(\(\phi'\)^2 - \phi \phi''\) + \phi\( \phi'\phi''-
  \phi\phi^{(3)}\)\). 
\end{equation*}
We claim that for all $r> 0$ (and $k\ge 1$), 
\begin{equation*}
 \phi'(r)\(\(\phi'(r)\)^2 - \phi (r)\phi''(r)\) > 1\quad ;\quad 
\phi(r)\( \phi'(r)\phi''(r) -
  \phi(r)\phi^{(3)}(r)\)> 0.
\end{equation*}
Since
\begin{equation*}
 \phi'\phi''-
  \phi\phi^{(3)} =  \(\(\phi'\)^2 - \phi \phi''\)',
\end{equation*}
and $\phi(0)=\phi''(0)=0$ and $\phi'(0)=1$, it suffices to show that
the above quantity is 
non-negative. From
\begin{equation*}
  \phi''(r)=\phi(r) -\frac{1}{(2k+1)!}r^{2k+1},\quad
  \phi^{(3)}(r)=\phi'(r) -\frac{1}{(2k)!}r^{2k} ,
\end{equation*}
we infer
\begin{align*}
 \Big(\phi'\phi''-
  \phi\phi^{(3)}\Big)(r)
&= -\frac{\phi'(r)}{(2k+1)!}r^{2k+1}+\frac{\phi(r)}{(2k)!}r^{2k}\\
&=\sum_{j=0}^{k}\frac{1}{(2j)!(2k)!}
\left(\frac{1}{2j+1}-\frac{1}{2k+1}\right)r^{2k+2j+1}>0. 
\end{align*}
The estimate announced above follows,
hence the lemma.   
\end{proof}
\begin{remark}
  If the function $\phi$ is replaced by 
  \begin{equation*}
    \phi(r) = \sum_{j=0}^k \frac{a_j}{(2j+1)!}r^{2j+1} 
  \end{equation*}
for some $a_j>0$, then the results of \cite{BD} show that weighted
Strichartz estimates are available in the radial setting, showing the
existence of wave operators with the same algebraic conditions as in
Theorem~\ref{theo:scattM}. However, for $k\ge 2$ and a general family
$(a_j)_{0\le  j\le k}$ of positive numbers, it is not clear whether the
analogue of the above 
lemma is valid or not: it may very well happen that with our choice for $a$,
$-\Delta^2_Ma$ has some zero for $0<r<\infty$, thus ruining the above
argument.  
\end{remark}

\section{Asymptotic completeness in hyperbolic space}
\label{sec:CAH}

In this paragraph, we prove Theorem~\ref{theo:CA}. 
\smallbreak

Since we are in a defocusing case, we have a global in time \emph{a
  priori} estimate for the $H^1$-norm of $u$, hence the following
  control, without radial assumption: 
\begin{equation}\label{globalM}
\left\|\sqrt\frac{\cosh
    r}{\sinh^3 r}u(t,x)\right\|_{L^2(\R,L^2(\H^n))}\le C(u_0). 
\end{equation}

This global control will allow us to prove that $u$ belongs globally
in time to certain weighted mixed spaces, yielding asymptotic
completeness. We set: 
\begin{equation*}
  {\tt w}_n={\tt w}_n(r)=\(\frac{\sinh r}{r}\)^{\frac{n-1}{2}},
\end{equation*}
and we denote by $d\Omega$ the measure on $\H^n$.
We recall that $(p,q)$ is  $n$-admissible if 
  \begin{equation}
\label{admissible}
\frac{2}{p}+\frac{n}{q}=\frac{n}{2},\quad p\ge 2, \quad
(p,q,n)\neq(2,\infty,2). 
\end{equation}
We shall use the following global Strichartz estimates  for the radial
free evolution, established in \cite{VittoriaDR} for $n\ge 4$ : 
\begin{equation}\label{wStrichartzfree}
    \left\|e^{it\Delta_{\H^n}}f(\cdot)
    \right\|_{L^p({\R},L^q({\tt w}_n^{q-2}d\Omega))}\le C \|f\|_{L^2},
\end{equation}
\begin{equation}\label{wStrichartzinhom}
      \left\|\int_{I\cap\{s\le
      t\}} e^{i(t-s)\Delta_{\H^n}}F(s)ds 
      \right\|_{L^{p}(I,L^{q}({\tt w}_n^{q-2}d\Omega))}\le C\left\|
      F\right\|_{L^{r'}\(I,L^{s'}({\tt w}_n^{s'-2}d\Omega)\)},
    \end{equation}
for all radial functions $f\in L_{\rm rad}^2(\H^n)$, $F\in
  L^{r'}\(I;L^{s'}_{\rm 
  rad}\(\H^n,{\tt w}_n^{s'-2}d\Omega\)\)$ and every $n$-admissible
  pairs $(p,q)$ and $(r,s)$. 
If $A$ is a derivative in
  space of order one,  similar estimates hold with the
  operator $A$ in front of $f$ and of 
  the integral in \eqref{wStrichartzinhom}. The constants are
  independent of the   time interval $I$.
\smallbreak

Let $A\in \{{\rm Id},\nabla\}$. In view of the above Strichartz
estimates, we wish to  control   
\begin{equation*}
  {\tt w}_n^{1-2/q'} A\(\lvert u\rvert^{2\si}u\)\text{ in }L^{p'}\(I;L^{q'}\)
\end{equation*}
by some power of $\|u\|_{X(I)}$, where 
\begin{equation}
  \label{eq:X(I)}
  \begin{aligned}
  X(I)=\Big\{& v\in L^\infty\(I,H^1(d\Omega)\)\cap L^{2}\(I,W^{1,2^*}({\tt
  w}_n^{2^*-2}d\Omega)\),\\
&\|v\|_{X(I)}=\|v\|_{L^\infty\(I,H^1(d\Omega)\)}+\|v\|_{L^{2}\(I,W^{1,2^*}({\tt
  w}_n^{2^*-2}d\Omega)\)} <\infty\Big\}, 
\end{aligned}
\end{equation}
and $2^*= \frac{2n}{n-2}$ (the pair $(2,2^*)$ is admissible, since
$n\ge 3$). This is achieved in
the following lemma.
\begin{lemma}\label{lem:algebre}
  Fix $n\ge 4$ and $0<\si<\frac{2}{n-2}$. Let $u$ be a radial solution
  to \eqref{nls}, and $A\in \{{\rm Id},\d_r\}$. There exist an 
  $n$-admissible pair $(p,q)$, $0<\alpha<2\si$, 
  and $C>0$ such that for all time interval $I$, 
  \begin{equation}
    \label{eq:estim}
    \left\lVert {\tt w}_n^{1-2/q'}A\(\lvert u\rvert^{2\si}u\)
\right\rVert_{L^{p'}(I;L^{q'})} \le C\lVert f \rVert_{L^1\(I\times
  \H^n\)}^{\alpha/2}\left\lVert 
  u\right\rVert_{X(I)}^{2\si+1-\alpha},
  \end{equation}
where $X(I)$ is defined in \eqref{eq:X(I)}, and
\begin{equation*}
  f(t,r) =\frac{\cosh r}{\sinh^3 r} \lvert u(t,r)\rvert^2. 
\end{equation*}
\end{lemma}
\begin{proof}
First, note that we have the uniform point-wise estimate
\begin{equation*}
  \left\lvert A\(\lvert u\rvert^{2\si}u\)\right\rvert \lesssim
   \lvert u\rvert^{2\si}  \left\lvert
  Au\right\rvert . 
\end{equation*}
  We want to apply H\"older's inequality, after the following splitting: 
\begin{multline}
  \label{eq:split}
  {\tt w}_n^{1-2/q'} \lvert u\rvert^{2\si}  \left\lvert
  Au\right\rvert = \( \sqrt{\frac{\cosh
  r}{\sinh^3 r} } \lvert u\rvert\)^\alpha 
\times \lvert u\rvert^{2\si -\alpha} \\
\times {\tt w}_n^{2/n}\left\lvert
  Au\right\rvert\,\times
{\tt w}_n^{1-2/q'}
{\tt w}_n^{-2/n}\(\frac{\sinh^3 r}{\cosh
  r}\)^{\alpha/2}. 
\end{multline}
The first term will be estimated in $L^{2/\alpha}_{t,x}$, the second
in $L^{\infty}_t L^{2^*/(2\si-\alpha)}_x$, the third in
$L^{2}_t L^{2^*}_x$, and the last in $L^\infty_t L^\theta_x$, where
$2^* = \frac{2n}{n-2}$. Write
\begin{equation}\label{eq:holder}
  \frac{1}{q'}=
  \frac{\alpha}{2}+
  \frac{2\si-\alpha}{2^*}
+ \frac{1}{2^*}+ \frac{1}{\theta}\quad
  ;\quad \frac{1}{p'}=\frac{\alpha}{2}
+\frac{1}{2}. 
\end{equation}
Thanks to Sobolev embedding, the third term will be controlled by 
\begin{equation*}
  \left\lVert u\right\rVert_{L^\infty_tL^{2*}_x}^{2\si-\alpha}
  \lesssim \left\lVert u\right\rVert_{L^\infty_t
  H^1_x}^{2\si-\alpha}\lesssim \left\lVert
  u\right\rVert_{X(I)}^{2\si-\alpha}. 
\end{equation*}
Since the pair $(2,2^*)$ is
$n$-admissible (endpoint), the lemma will follow if we can choose
$(p,q)$ and $\alpha$ such that:
\begin{itemize}
\item $(p,q)$ is $n$-admissible.
\item $\alpha>0$, with $2\si-\alpha>0$. (Morally, $0<\alpha\ll 1$.)
\item The last factor in \eqref{eq:split} is in $L^\theta(\H^n)$,
  for some $\theta\in [1,\infty[$. 
\end{itemize}
If $p$ is imposed in view of \eqref{eq:holder}, that is
\begin{equation*}
  \frac{1}{p}= \frac{1}{2}-\frac{\alpha}{2},
\end{equation*}
then $(p,q)$ is $n$-admissible if
\begin{equation}
  \label{eq:q}
  \frac{1}{q} = \frac{1}{2^*}+\frac{\alpha}{n}. 
\end{equation}
This is consistent with the first equality of \eqref{eq:holder}
provided that
\begin{equation}\label{eq:theta}
  \frac{1}{\theta}= \frac{2}{n} -\frac{n-2}{n}\si -\frac{2\alpha}{n}. 
\end{equation}
Examine the last condition of the three listed above. Working
in radial coordinates, 
recall that the 
measure element is $\sinh^{n-1}r dr$. 
Integrability near $r=0$ is not a
problem. Integrability as $r\to \infty$
follows from an exponential decay, provided that: 
\begin{equation*}
 \theta\(\alpha + 
\frac{n-1}{2}\(1-\frac{2}{q'}-\frac{2}{n}
\)\)+n-1<0 .
\end{equation*}
Using \eqref{eq:q} and \eqref{eq:theta}, this becomes:
\begin{equation}\label{eq:integrabilite}
  \frac{\alpha}{n-1}-\frac{\alpha}{2} -\frac{2\si-\alpha}{2^*}<0.
\end{equation}
Consider the extreme case $\alpha=0$. The above condition is obviously
fulfilled, and $\theta$ is finite since $\si<\frac{2}{n-2}$, with
$\theta>1$ since $\si>0$. 
\smallbreak

Since the conditions  $\theta \in ]1,\infty[$ and
\eqref{eq:integrabilite} are open, by continuity, we can
find $\alpha>0$ such that they remain valid. 
So we have fulfilled all the  conditions listed above, and the lemma
follows from H\"older's inequality. 
\end{proof}
\begin{proof}[Proof of Theorem~\ref{theo:CA}]
Let $I$ be some time interval, and $t_0\in I$. For $(p,q)$ the
$n$-admissible pair of Lemma~\ref{lem:algebre}, weighted Strichartz
estimates \eqref{wStrichartzfree}--\eqref{wStrichartzinhom} yield
\begin{equation*}
  \left\lVert u\right\rVert_{X(I)} \le C\(\left\lVert
  u(t_0)\right\rVert_{H^1(\H^n)}  +\lVert f \rVert_{L^1\(I\times
  \H^n\)}^{\alpha/2}\left\lVert 
  u\right\rVert_{X(I)}^{2\si+1-\alpha}\). 
\end{equation*}
We have seen in the proof of Lemma~\ref{lem:algebre} that $\alpha>0$ is
such that $2\si>\alpha$. Therefore, the exponent
$2\si+1-\alpha$ is larger than one. 
Recall the standard bootstrap argument (see e.g. \cite{BG3}). 
\begin{lemma}[Bootstrap argument]\label{lem:boot}
Let $M=M(t)$ be a nonnegative continuous function on $[0,T]$ such
that, for every $t\in [0,T]$, 
\begin{equation*}
  M(t)\le \eps_1 + \eps_2 M(t)^\theta,
\end{equation*}
where $\eps_1,\eps_2>0$ and $\theta >1$ are constants such that
\begin{equation*}
  \eps_1 <\left(1-\frac{1}{\theta} \right)\frac{1}{(\theta \eps_2)^{1/(\theta
-1)}}\ ,\ \ \ M(0)\le  \frac{1}{(\theta \eps_2)^{1/(\theta
-1)}}.
\end{equation*}
Then, for every $t\in [0,T]$, we have
\begin{equation*}
  M(t)\le \frac{\theta}{\theta -1}\ \eps_1.
\end{equation*}
\end{lemma}
Let $\eps>0$. Since $f\in L^1(\R\times\H^n)$, we can split $\R_+$ into
a \emph{finite} family 
\begin{equation*}
  \R_+ = \bigcup_{j=1}^J I_j,\quad I_j=[T_j,T_{j+1}[,\text{with }
  T_1=0\text{ and }
  T_{J+1}=+\infty,
\end{equation*}
so that $ \lVert f\rVert_{L^1(I_j\times\H^n)}\le \eps.$ 
Choosing $\eps>0$ sufficiently small and summing up
  over the $I_j$'s, we conclude:
  \begin{equation*}
   u\in X(\R). 
  \end{equation*}
Using weighted Strichartz inequality again, we see that
$\(e^{-it\Delta_{\H^n}} u(t,\cdot)\)_{t>0}$ is a 
Cauchy sequence in $H^1(\H^n)$ as $t\to +\infty$.
So, there is scattering at the $H^1$ level:
\begin{equation*}
  \exists u_+\in H^1\(\H^n\),\quad \left\lVert
  u(t)-e^{it\Delta_{\H^n}}u_+\right\rVert_{H^1(\H^n)} \Tend t {+\infty}0 .
\end{equation*}
This completes the proof of Theorem~\ref{theo:CA}. In view of
\cite{BCS}, 
Corollary~\ref{cor:scatt} follows. 
\end{proof}

\section{Asymptotic completeness in intermediary manifolds}
\label{sec:scattM}

The proof of Theorem~\ref{theo:scattM} follows the same strategy as
above. Introduce 
\begin{equation*}
 {\tt w}_n= {\tt w}_n(r)= \(\frac{\phi(r)}{r}\)^{\frac{n-1}{2}}, 
\end{equation*}
and denote by $d\Omega$ the measure on $M_k^n$. The following weighted
Strichartz estimates are established in \cite{BD}:
\begin{align*}
    \left\lVert e^{it\Delta_{M}}f(\cdot)
    \right\rVert_{L^p({\R},L^q({\tt w}_n^{q-2}d\Omega))}&\le C \|f\|_{L^2},\\
      \left\|\int_{I\cap\{s\le
      t\}} e^{i(t-s)\Delta_{M}}F(s)ds 
      \right\|_{L^{p}(I,L^{q}({\tt w}_n^{q-2}d\Omega))}&\le C\left\|
      F\right\|_{L^{r'}\(I,L^{s'}({\tt w}_n^{s'-2}d\Omega)\)},
    \end{align*}
for all radial functions $f\in L_{\rm rad}^2(M)$, $F\in
  L^{r'}\(I;L^{s'}_{\rm 
  rad}\(\H^n,{\tt w}_n^{s'-2}d\Omega\)\)$ and every $n$-admissible pairs $(p,q)$ and $(r,s)$.
If $A$ is a derivative in
  space of order one,  similar estimates hold with the
  operator $A$ in front of $f$ and of 
  the above retarded integral. The constants are
  independent of the   time interval $I$. Mimicking the proof of
  Theorem~\ref{theo:CA}, it suffices to prove the following
\begin{lemma}\label{lem:algebre2}
  Fix $n\ge 4$, $k\ge 1$ and $2/N<\si<\frac{2}{n-2}$, where
  $N=(2k+1)(n-1)+1$. Let $u$ be a radial solution
  to \eqref{eq:nlsM}, and $A\in \{{\rm Id},\d_r\}$. There exist an 
  $n$-admissible pair $(p,q)$, $0<\alpha<2\si$, 
  and $C>0$ such that for all time interval $I$, 
  \begin{equation}
    \label{eq:estimM}
    \left\lVert {\tt w}_n^{1-2/q'}A\(\lvert u\rvert^{2\si}u\)
\right\rVert_{L^{p'}(I;L^{q'})} \le C\lVert f \rVert_{L^1\(I\times
  \H^n\)}^{\alpha/2}\left\lVert 
  u\right\rVert_{X(I)}^{2\si+1-\alpha},
  \end{equation}
where $X(I)$ is defined as in \eqref{eq:X(I)}, and
\begin{equation*}
  f(t,r) =\frac{1}{r^3} \lvert u(t,r)\rvert^2. 
\end{equation*}
\end{lemma}
\begin{proof}
  The proof is very similar to the proof of Lemma~\ref{lem:algebre}.
  We want to apply H\"older's inequality, after the following splitting:
\begin{equation*}
  \begin{aligned}
  {\tt w}_n^{1-2/q'} \lvert u\rvert^{2\si}  \left\lvert
  Au\right\rvert = &\( \frac{1}{r^{3/2}} \lvert u\rvert\)^\alpha 
\times\lvert u\rvert^{2\si -\alpha}
\times {\tt w}_n^{2/n}\left\lvert
  Au\right\rvert\times
{\tt w}_n^{1-2/q'}
{\tt w}_n^{-2/n} r^{3\alpha/2}. 
\end{aligned}
\end{equation*}
Write
\begin{equation}\label{eq:holderM}
  \frac{1}{q'}=
  \frac{\alpha}{2}+
  \frac{2\si-\alpha}{a}
+ \frac{1}{2^*}+ \frac{1}{\theta}\quad
  ;\quad \frac{1}{p'}=\frac{\alpha}{2}
+\frac{1}{2}, 
\end{equation}
where $2^* = \frac{2n}{n-2}$. The lemma will follow if we can choose
$(p,q)$, $a$, and $\alpha$, such that:
\begin{itemize}
\item $(p,q)$ is $n$-admissible.
\item $\alpha>0$, with $2\si-\alpha>0$. 
\item $a\in [2,2^*]$ (to control the second term thanks to
  Sobolev embedding). 
\item The last factor in the above splitting is in $L^\theta(M_k^n)$,
  for some $\theta\in [1,\infty[$. 
\end{itemize}
With $p$ given by the second equation in \eqref{eq:holderM}, the first
condition is equivalent to 
\begin{equation}
  \label{eq:qM}
  \frac{1}{q} = \frac{1}{2^*}+\frac{\alpha}{n}. 
\end{equation}
This is consistent with the first equality of \eqref{eq:holderM}
provided that
\begin{equation}\label{eq:thetaM}
  \frac{1}{\theta}= \frac{2}{n} -\frac{2\si}{a}
  -\alpha\(\frac{1}{a}+\frac{1}{n}+\frac{1}{2}\). 
\end{equation}
Let us examine the last condition of the four listed above. Working
in radial coordinates, 
recall that the 
measure element is $\phi(r)^{n-1} dr$. 
Integrability near $r=0$ is not a
problem. Integrability as $r\to \infty$ holds if: 
\begin{equation*}
 \theta\(\frac{3\alpha}{2} + 
k(n-1)\(1-\frac{2}{q'}-\frac{2}{n}
\)\)+(2k+1)(n-1)<-1,
\end{equation*}
that is, thanks to \eqref{eq:qM},
\begin{equation*}
  \theta\( \frac{3\alpha}{2} +  \frac{N-n}{n}\alpha
  -\frac{2(N-n)}{n}\)+N<0\Longleftrightarrow \frac{1}{\theta} <
  \frac{2}{n}-\frac{2}{N} -\alpha\(\frac{1}{2N}+\frac{1}{n}\). 
\end{equation*}
Consider the extreme case $\alpha=0$.
The above condition on $\theta$ yields
\begin{equation*}
  \frac{2}{N}<\frac{2\si}{a} \Longleftrightarrow \frac{a}{N}<\si.
\end{equation*}
In view of \eqref{eq:thetaM}, $\theta$ is finite if
\begin{equation*}
  \si<\frac{a}{n}. 
\end{equation*}
Moreover $\theta$ is always larger than $1$ ($n\ge 4$ and $\si>0$). 
Since $k\ge 1$, $N>n$, and since $\frac{2}{N}<\si<\frac{2}{n-2}$, we
can find $a\in [2,2^*]$ such that 
\begin{equation*}
  \frac{a}{N}<\si<\frac{a}{n}. 
\end{equation*}
Fix the parameter $a$. Since the requirements we have made are open
conditions, by continuity, we can find $0<\alpha\ll 1$ such that they
are still satisfied.
So we have fulfilled all the  conditions listed above, and the lemma
follows from H\"older's inequality. 
\end{proof}

\section{The free dynamics in the radial case}
\label{sec:free}

To prove Proposition~\ref{prop:free}, we first reduce the analysis to
the Euclidean case, as in \cite{VittoriaDR,BD}. Consider the equation
\begin{equation} \label{eq:17h43}
  i\d_t u+\Delta u =0 ,
\end{equation}
where $\Delta$ stands for the Laplace--Beltrami associated to an
$n$-dimensional rotationally symmetric manifold with metric
\begin{equation*}
  ds^2 = dr^2 + \phi(r)^2d\omega^2,\text{ where }\phi(r) = \sum_{j=0}^k
  \frac{1}{(2j+1)!} r^{2j+1},
\end{equation*}
and $k$ is possibly infinite. 
Introduce $\widetilde
u$, given by
\begin{equation*}
  \widetilde u (t,r) = u(t,r) \(\frac{\phi(r)}{r}\)^{\frac{n-1}{2}}. 
\end{equation*}
In the case of radial solutions, \eqref{eq:17h43} is equivalent to
\begin{align*}
  &i\d_t \widetilde u + \Delta_{\R^n} \widetilde u = V \widetilde
  u,\\
\text{where }&V(r)=
  \frac{n-1}{2}\frac{\phi''(r)}{\phi(r)}+\frac{(n-1)(n-3)}{4}
  \(\(\frac{\phi'(r)}{\phi(r)}\)^2 -\frac{1}{r^2}\). 
\end{align*}
We check easily the following dichotomy:
\begin{itemize}
\item If $k$ is finite, then $V$ is smooth, with $V(r)=\O(r^{-2})$ as
  $r\to \infty$. 
\item If $k=\infty$ (case of hyperbolic space), then $\phi''=\phi$,
  and $V= (n-1)/2 +\widetilde V$, where $\widetilde V$ is as above. 
\end{itemize}
Up
  to replacing $\widetilde u$ with $e^{i(n-1)t/2}\widetilde u$ when
  $k$ is infinite, we see that it suffices to study the first
  case. The potential $V$ is a short range potential, as far as linear
  scattering theory is concerned (see
  e.g. \cite{DG,Yafaev}). Therefore, there exists $\widetilde u_+\in
  L^2_{\rm rad}(\R^n)$ such that
  \begin{equation*}
    \left\lVert \widetilde u(t) - e^{it\Delta_{\R^n}}\widetilde
    u_+\right\rVert_{L^2(\R^n)} \Tend t{+\infty} 0. 
  \end{equation*}
Moreover, the map $\widetilde u_{\mid t=0}\mapsto \widetilde u_+$ is
linear and continuous from $L^2_{\rm rad}(\R^n)$ to $L^2_{\rm
  rad}(\R^n)$. Recalling that the volume element is $r^{n-1}dr$ on
$\R^n$, and $\phi(r)^{n-1}dr$ on the manifold that we consider, we infer
\begin{equation*}
  \left\lVert u(t) - v_1(t)\right\rVert_{L^2(M)} \Tend t{+\infty} 0,
\end{equation*}
where
\begin{equation*}
  v_1(t,r)= \(\frac{r}{\phi(r)}\)^{\frac{n-1}{2}}e^{it\Delta_{\R^n}}\(
   \widetilde u_+(r)\). 
\end{equation*}
Proposition~\ref{prop:free} then follows from  the standard large time
asymptotics for $e^{it\Delta_{\R^n}}$,
\begin{equation}\label{eq:DAeucl}
  \left\lVert e^{it\Delta_{\R^n}} \varphi - \Lambda
  (t)\right\rVert_{L^2(\R^n)} \Tend t {+\infty}0,
\text{ where } \Lambda(t,x) = \frac{e^{i\lvert
  x\rvert^2/(4t)}}{t^{n/2}} \(\F \varphi\)\(\frac{x}{2t}\), 
\end{equation}
and the Fourier transform $\F$ is normalized so the above relation
holds true. This asymptotics is, for instance,  a straightforward
consequence of 
the factorization  $e^{it\Delta_{\R^n}}=MD\F M$, where $M$ is the
multiplication by an exponential, and $D$ is the $L^2$-unitary
dilation at scale $t$.  
\begin{remark}\label{rem:Fourier}
  We notice that $V=0$ if $\phi'$ is constant: in the Euclidean case,
  $\widetilde u = c u$. If $n=3$ and $\phi''=c\phi$, $V$ is constant:
  for radial solutions on $\H^3$, and up to a purely time dependent
  phase shift, there is no external potential. In the two cases
  distinguished here, we have $\widetilde u_+=\widetilde u_{\mid
  t=0}$. Then \eqref{eq:DAeucl} shows why the Fourier transform is
  present in the description of the asymptotic behavior of radial
  solutions to \eqref{eq:17h43}. Recall that the Fourier transform for
  radially symmetric functions on $\H^n$
  is much simpler when $n=3$; see \cite{BaHyper} and references
  therein. 
\end{remark}
\begin{remark}\label{rem:long}
  Following the formal argument given in \cite{Ginibre},
  Proposition~\ref{prop:free} suggests that for $\si\le 1/N$, long
  range effects are present in \eqref{eq:nlsM}. Suppose that $n\ge 2$
  and 
  $0<\si\le 1/N$, where $N=(2k+1)(n-1)+1$. Let $u\in
  C([T,\infty[;L^2_{\rm rad}(M_k^n))$ be a solution of \eqref{eq:nlsM}
  such that there exists $u_+\in L^2_{\rm rad}(M_k^n)$ with
  \begin{equation*}
    \left\lVert u(t)- e^{it\Delta_M}u_+\right\rVert_{L^2}\Tend t
    {+\infty}0.
  \end{equation*}
Formal computations indicate that necessarily, $u_+\equiv 0$ and
$u\equiv 0$: the linear and nonlinear dynamics are no longer
comparable, due to long range effects. To see this, let $\psi
\in C_0^\infty(M)$ be radial, and $t_2\ge t_1\ge T$. By assumption, 
  \begin{equation*}
    \< \psi, e^{-it_2\Delta_M}u(t_2)-e^{-it_1\Delta_M}u(t_1)\>=
    -i\int_{t_1}^{t_2}\< 
    e^{it\Delta_M}\psi, \(|u|^{2\si}u\)(t)\>dt
  \end{equation*}
goes to zero as $t_1,t_2\to +\infty$. Proposition~\ref{prop:free}
implies that for $t\to +\infty$, we have
\begin{equation*}
  \< 
    e^{it\Delta_M}\psi, \(|u|^{2\si}u\)(t)\> \approx
    \frac{1}{t^{n\si+n}} \int_0^\infty\(\frac{r}{\phi(r)}\)^{(n-1)(\si+1)}
    \varphi\(\frac{r}{t}\)\phi(r)^{n-1}dr, 
\end{equation*}
for $\varphi = \L \psi \lvert \L u_+\rvert^{2\si}\overline{\L
  u_+}$. With the change of variable $r\mapsto t r$, the above
  integral is equal to 
  \begin{equation*}
    \frac{1}{t^{n\si+n-1}} \int_0^\infty\(\frac{t r }{\phi(t
    r)}\)^{(n-1)(\si+1)} 
    \varphi\(r\)\phi(t r)^{n-1}dr. 
  \end{equation*}
For $r\ge 1$ and large $t$, the function at stake behaves like
\begin{equation*}
  \frac{1}{t^{n\si+n-1}} \(\frac{t r }{(t
    r)^{2k+1}}\)^{(n-1)(\si+1)}\varphi(r)\(tr\)^{(n-1)(2k+1)} =
    \frac{r^{-(N-n)\si+n-1}}{t^{N\si}}\varphi(r). 
\end{equation*}
This function of $t$ is not integrable, unless $\varphi\equiv 0$.
This means that
$\L u_+= 0=u_+$ ($\operatorname{Ker}\L =\{0\}$). The assumption and the
conservation of mass then imply $u\equiv 0$.
\end{remark}
\bibliographystyle{amsplain}
\bibliography{hyper}

\providecommand{\bysame}{\leavevmode\hbox to3em{\hrulefill}\thinspace}
\providecommand{\MR}{\relax\ifhmode\unskip\space\fi MR }
\providecommand{\MRhref}[2]{%
  \href{http://www.ams.org/mathscinet-getitem?mr=#1}{#2}
}
\providecommand{\href}[2]{#2}
\begin{thebibliography}{10}

\bibitem{AnPi}
J.-P. Anker and V.~Pierfelice, in preparation.

\bibitem{BG3}
H.~Bahouri and P.~G{\'e}rard, \emph{High frequency approximation of solutions
  to critical nonlinear wave equations}, Amer. J. Math. \textbf{121} (1999),
  no.~1, 131--175.

\bibitem{BGX}
H.~Bahouri, P.~G{\'e}rard, and C.-J. Xu, \emph{Espaces de {B}esov et
  estimations de {S}trichartz g\'en\'eralis\'ees sur le groupe de
  {H}eisenberg}, J. Anal. Math. \textbf{82} (2000), 93--118.

\bibitem{BaHyper}
V.~Banica, \emph{The nonlinear {S}chr\"odinger equation on the hyperbolic
  space}, Comm. Partial Differential Equations \textbf{32} (2007), no.~10,
  1643--1677.

\bibitem{BCS}
V.~Banica, R.~Carles, and G.~Staffilani, \emph{Scattering theory for radial
  nonlinear {S}chr\"odinger equations on hyperbolic space}, Geom. Funct. Anal.
  (2008), to appear.

\bibitem{BD}
V.~Banica and T.~Duyckaerts, \emph{Weighted {S}trichartz estimates for radial
  {S}chr\"odinger equation on noncompact manifolds}, archived as {\tt
  arXiv:0707.3370}, 2007.

\bibitem{Barab}
J.~E. Barab, \emph{Nonexistence of asymptotically free solutions for nonlinear
  {S}chr\"odinger equation}, J. Math. Phys. \textbf{25} (1984), 3270--3273.

\bibitem{BCHM}
J.-F. Bony, R.~Carles, D.~H\"afner, and L.~Michel, \emph{Scattering theory for
  the {S}chr\"odinger equation with repulsive potential}, J. Math. Pures Appl.
  \textbf{84} (2005), no.~5, 509--579.

\bibitem{CaSIMA}
R.~Carles, \emph{Nonlinear {S}chr\"odinger equations with repulsive harmonic
  potential and applications}, SIAM J. Math. Anal. \textbf{35} (2003), no.~4,
  823--843.

\bibitem{CazCourant}
T.~Cazenave, \emph{Semilinear {S}chr\"odinger equations}, Courant Lecture Notes
  in Mathematics, vol.~10, New York University Courant Institute of
  Mathematical Sciences, New York, 2003.

\bibitem{CKSTTCPAM}
J.~Colliander, M.~Keel, G.~Staffilani, H.~Takaoka, and T.~Tao, \emph{Global
  existence and scattering for rough solutions of a nonlinear {S}chr\"odinger
  equation on {$\mathbb R\sp 3$}}, Comm. Pure Appl. Math. \textbf{57} (2004),
  no.~8, 987--1014.

\bibitem{CKSTTAnnals}
\bysame, \emph{Global well-posedness and scattering for the energy--critical
  nonlinear {S}chr\"odinger equation in {$\mathbb R\sp 3$}}, Ann. of Math. (2)
  (2008), to appear.

\bibitem{DG}
J.~Derezi\'nski and C.~G\'erard, \emph{Scattering theory of quantum and
  classical {N}-particle systems}, Texts and Monographs in Physics, Springer
  Verlag, Berlin Heidelberg, 1997.

\bibitem{Dunford}
N.~Dunford and J.~T. Schwartz, \emph{Linear operators. {P}art {II}: {S}pectral
  theory. {S}elf adjoint operators in {H}ilbert space}, With the assistance of
  William G. Bade and Robert G. Bartle, Interscience Publishers John Wiley \&
  Sons\ New York-London, 1963.

\bibitem{Ginibre}
J.~Ginibre, \emph{An introduction to nonlinear {S}chr\"odinger equations},
  Nonlinear waves (Sapporo, 1995) (R.~Agemi, Y.~Giga, and T.~Ozawa, eds.),
  GAKUTO International Series, Math. Sciences and Appl., Gakk\={o}tosho, Tokyo,
  1997, pp.~85--133.

\bibitem{GV79Cauchy}
J.~Ginibre and G.~Velo, \emph{On a class of nonlinear {S}chr\"odinger
  equations. {I} {T}he {C}auchy problem, general case}, J. Funct. Anal.
  \textbf{32} (1979), 1--32.

\bibitem{GV85}
\bysame, \emph{Scattering theory in the energy space for a class of nonlinear
  {S}chr\"odinger equations}, J. Math. Pures Appl. (9) \textbf{64} (1985),
  no.~4, 363--401.

\bibitem{HTW05}
A.~Hassell, T.~Tao, and J.~Wunsch, \emph{A {S}trichartz inequality for the
  {S}chr\"odinger equation on nontrapping asymptotically conic manifolds},
  Comm. Partial Differential Equations \textbf{30} (2005), no.~1-3, 157--205.

\bibitem{IS08}
A.~Ionescu and G.~Staffilani, \emph{Semilinear {S}chr\"odinger flows on
  hyperbolic spaces: scattering in ${H}^1$}, archived as {\tt arXiv:0801.2957},
  2008.

\bibitem{Pe}
P.~Petersen, \emph{Riemannian geometry}, Graduate Texts in Math., vol. 171,
  Springer, 1997.

\bibitem{VittoriaDR}
V.~Pierfelice, \emph{Weighted {S}trichartz estimates for the {S}chr\"odinger
  and wave equations on {D}amek-{R}icci spaces}, Math. Z. (2008), to appear.

\bibitem{RV}
E.~Ryckman and M.~Visan, \emph{Global well-posedness and scattering for the
  defocusing energy--critical nonlinear {S}chr\"odinger equation in {$\mathbb
  R\sp{1+4}$}}, Amer. J. Math. \textbf{129} (2007), no.~1, 1--60.

\bibitem{Strauss74}
W.~A. Strauss, \emph{Nonlinear scattering theory}, Scattering theory in
  mathematical physics (J.~Lavita and J.~P. Marchand, eds.), Reidel, 1974.

\bibitem{TaoVisanZhang}
T.~Tao, M.~Visan, and X.~Zhang, \emph{The nonlinear {S}chr\"odinger equation
  with combined power-type nonlinearities}, Comm. in Partial Diff. Eq.
  \textbf{32} (2007), 1281--1343.

\bibitem{VisanH1}
M.~Visan, \emph{The defocusing energy-critical nonlinear {S}chr\"odinger
  equation in higher dimensions}, Duke Math. J. \textbf{138} (2007), no.~2,
  281--374.

\bibitem{Yafaev}
D.~Yafaev, \emph{Scattering theory: some old and new problems}, Lecture Notes
  in Mathematics, vol. 1735, Springer-Verlag, Berlin, 2000.

\end{thebibliography}

\end{document}